\documentclass[11pt,leqno]{amsart}
\usepackage{amsmath,amssymb,txfonts}
\usepackage{amssymb}
\usepackage{amsxtra}
\usepackage{amsmath}
\usepackage{txfonts}
\usepackage[notref,notcite]{showkeys}

 \usepackage{color}

\usepackage[cp1252]{inputenc}

\usepackage{mathrsfs}
\usepackage{graphicx}

\usepackage[active]{srcltx}

\allowdisplaybreaks

\def\rr{{\mathbb R}}
\def\rn{{{\rr}^n}}

\def\nn{{\mathbb N}}

\def\fz{\infty}

\def\dist{{\mathop\mathrm{\,dist\,}}}
\def\loc{{\mathop\mathrm{\,loc\,}}}

\def\dz{\delta}

\def\ez{\epsilon}

\def\gz{{\gamma}}

\def\boz{{\Omega}}

\def\vz{\varphi}

\def\wz{\widetilde}

\def\diam{{\mathop\mathrm{\,diam\,}}}

\def\r{\right}
\def\lf{\left}

\newtheorem{thm}{Theorem}[section]
\newtheorem{lem}[thm]{Lemma}

\newtheorem{rem}[thm]{Remark}
\newtheorem{cor}[thm]{Corollary}

\numberwithin{equation}{section}

\begin{document}
\arraycolsep=1pt

\title[Orlicz-Besov extension and Ahlfors $n$-regular domains ]{Orlicz-Besov extension and Ahlfors $n$-regular domains}

\author{Tian Liang and Yuan Zhou}

          \address{ Department of Mathematics, Beihang University, Beijing 100191, P.R. China}
                    \email{614895606@qq.com, yuanzhou@buaa.edu.cn}

\thanks{    }

\date{\today }
\maketitle

\begin{center}
\begin{minipage}{12cm}\small{\noindent{\bf Abstract}\quad
Let $n\ge2$ and  $\phi : [0,\fz) \to [0,\infty)$ be a
 Young's function satisfying $\sup_{x>0}
\int_0^1\frac{\phi( t x)}{ \phi(x)}\frac{dt}{t^{n+1} }<\infty.
$
 We show that  Ahlfors $n$-regular domains are Besov-Orlicz ${\dot {\bf B}}^{\phi}$ extension domains,
 which is necessary to guarantee the nontrivially of ${\dot {\bf B}}^{\phi}$.
On the other hand, assume that $\phi$ grows sub-exponentially at $\fz$ additionally. If $\Omega$ is a
Besov-Orlicz ${\dot {\bf B}}^{\phi}$ extension domain, then it must be Ahlfors $n$-regular.
}
\end{minipage}
\end{center}

\section{Introduction\label{s1}}

Let $\phi : [0,\fz) \to [0,\infty)$  be a Young  function, that is,  $\phi$ is a convex,
 $\phi(0)=0$, $\phi(t)>0$ for $t>0$ and $\lim  _{t \to \infty} \phi (t)= +\infty$.
 Given any domain $\Omega\subset\rn$,  the Orlicz-Besov space   ${\dot {\bf B}}^\phi(\Omega)$   consists of
  all measurable functions $u $ in $\Omega$  whose (semi-)norms
\begin{eqnarray*}
\|u\|_{{\dot {\bf B}}^{\phi}(\Omega)} := \inf \left\{\alpha > 0 : \int_{\Omega} \int_{\Omega}  \phi \left(\frac{|u(x)-u(y)|}{\alpha}\right) \frac{dxdy}{|x-y|^{2n}}\leq 1 \right\}
\end{eqnarray*}
is finite.   Modulo constant functions, ${\dot {\bf B}}^{\phi}(\Omega)$ is a Banach space. We refer to \cite{mcp} for the applications of Orlicz-Besov spaces in  qausi-conformal geometry.
 Note that, in the case of   $ \phi(t)=t^p$ with $p\ge1$,   the  ${\dot {\bf B}}^{\phi}(\Omega)$-norms are written as
 \begin{align}\label{eq1.t1}
\|u\|_{{\dot {\bf B}}^\phi(\Omega)} =   \lf(\int_{\Omega} \int_{\Omega}  \frac{|u(x)-u(y)|^p}{|x-y|^{2n}} \,dxdy\r)^{1/p}.
\end{align}
By this,   when  $ \phi(t)=t^p$ with $p>n$,
$  \dot {\bf B} ^\phi(\Omega)$ is exactly the Besov spaces $  {\dot {\bf B}}^{n/p}_{p,p}(\boz)$ (or
 fractional Sobolev spaces $\dot W^{n/p,p}(\Omega) $).
 However, when $ \phi(t)=t^p$ with $p\le n$,
   thanks to \eqref{eq1.t1} and \cite{gkz13},
   the space ${\dot {\bf B}}^{\phi}(\Omega)$  is  trivial, that is, only contains constant functions.

 In general, to guarantee the nontrivially of ${\dot {\bf B}}^\phi(\Omega)$,  we always assume
\begin{equation}\label{delta0}C_\phi :=\sup_{x>0}
\frac{t^n}{\phi(t)}\int_0^t\frac{\phi( s  )}{ s^n}\frac{ds}{s }<\infty.
\end{equation}
Indeed,  \eqref{delta0} does imply that  $  {\dot {\bf B}}^\phi(\Omega)$ contains smooth functions with compact supports, and hence nontrivial; see Lemma \ref{l3.3} below.
If  $\phi(t)=t^p$, observe that  $\phi$  satisfies \eqref{delta0} if and only if  $p>n$,
and ${\dot {\bf B}}^\phi(\Omega)$ is nontrivial if and only if $p>n$.
In this sense, we see that \eqref{delta0} is optimal  to guarantee the nontrivially of ${\dot {\bf B}}^\phi(\Omega)$.
There are some other interesting
 Young functions satisfying \eqref{delta0},  for example, $t^n[\ln(1+t)]^\alpha$ with   $\alpha>1$,
$t^p[\ln(1+t)]^\alpha$ with $p> n$ and $\alpha\ge 1$, $t^pe^{ct^\alpha}$  with $p> n$, $c>0$ and $\alpha> 0$,
and   $e^{c t^\alpha }-\sum_{j=0}^{[n/\alpha] }(ct^\alpha)^j/j!$ where  $c>0$ with $\alpha>0$,
 where $[n/\alpha]$ is the maximum of integers no bigger than $n/\alpha$.

In this paper, we obtain the following results for the Orlicz-Besov extension in
  Ahlfors $n$-regular domains.
Recall that a domain $\Omega$ is  Ahlfors $n$-regular  if there exists a constant $C_A(\Omega)>0$ such that
$$
|B(x,r)\cap\Omega|\ge C_A(\Omega) r^n\quad\forall x\in\Omega, 0<r<2\diam\Omega.$$
A domain $\Omega$ is called ${\dot {\bf B}}^\phi$-extension domain if  any function
$u \in {\dot {\bf B}}^\phi(\Omega)$ can be extended to as a function $\wz u\in {\dot {\bf B}}^\phi(\mathbb R^n)$
continuously  and linearly;
in other words,  there exists a  bounded linear operator $E: {\dot {\bf B}}^\phi(\Omega)\to {\dot {\bf B}}^\phi(\rn)$ with $Eu|_\Omega=u$ for any $u\in {\dot {\bf B}}^\phi(\Omega)$.

\begin{thm}\label{t1.1}
 Let   $\phi$ be a Young function satisfying  \eqref{delta0}.
 \begin{enumerate}
\item[(i)]  If   $\Omega\subset\rn$ is  Ahlfors $n$-regular domain,
then    $\Omega$ is a ${\dot {\bf B}}^\phi$-extension domain.

\item[(ii)] Assume that $\phi$ additionally satisfies  \begin{equation}\label{subexp}
\limsup_{x\to\infty}  { \phi(x)}{e^{-cx}}=0\quad \forall c>0.
\end{equation} If
   $\Omega\subset\rn$ is a ${\dot {\bf B}}^\phi$-extension domain, then  $\Omega$ is  Ahlfors $n$-regular.
   \end{enumerate}
\end{thm}

Note that the condition  \eqref{subexp} in Theorem 1.1 (ii)  allows a large class of Young functions,
 including $t^n[\ln(1+t)]^\alpha$ with   $\alpha>1$,
$t^p[\ln(1+t)]^\alpha$ with $p> n$ and $\alpha\ge 1$, $t^pe^{ct^\alpha}$  with $p> n$ and $\alpha\in(0,1)$,
and   $e^{c t^\alpha }-\sum_{j=0}^{[n/\alpha] }(ct^\alpha)^j/j!$ where  $c>0$ with $\alpha\in(0,1)$.
 But \eqref{subexp} rules out  $t^pe^{ct^\alpha}$  with $p> n$ and $\alpha\ge 1$,
and   $e^{c t^\alpha }-\sum_{j=0}^{[n/\alpha] }(ct^\alpha)^j/j!$ where  $c>0$ with $\alpha\ge1$,

Theorem 1.1  extends the known results for fractional Sobolev  spaces $\dot {\bf W}^{n/p,p}(\Omega)$ or
Besov space $\dot {\bf B}^{n/p}_{p,p}(\Omega)$.
Recall that the extension problem for function spaces (including Sobolev, fracntional Sobolev, Hajlasz-Sobolev,
 Besov and Triebel-Lizorkin spaces)
have been widely studied in  the literature,
see \cite{j80,j81,s70,jw78,jw84,ds93,k98,r99,t02,t08,hkt08,s06,s10}
and the references therein.
Given function spaces $X(U)$ defined in any domain   $U\subset \rn$  in the same manner,
define  $X$-extension domains  similarly to  $\dot{\bf B}^\phi$-extension domains.
It turns out that the extendability  of functions in $X(\Omega)$
 not only relies on the geometry of the domain
 but also on the analytic properties of $X$.
In particular, it was essentially known that  Ahlfors $n$-regular domains are
fractional Sobolev $\dot {\bf W}^{s,p} $-extension domains  for any $s\in(0,1)$ and  $p\ge1$;
see Jonsson-Wallin \cite{jw78} (also Shvartsman \cite{s06}).
Here $\dot {\bf W}^{s,p} (\Omega)$ is the set of all functions with  \begin{align}
\|u\|_{{\dot {\bf W}}^{s,p}(\Omega)}: =   \lf(\int_{\Omega} \int_{\Omega}  \frac{|u(x)-u(y)|^p}{|x-y|^{ n+sp}} \,dxdy\r)^{1/p}<\fz.
\end{align}
Moreover, by  Shvartsman \cite{s07} and Haj\l asz et al \cite{hkt08,hkt08b}, Ahlfors $n$-regular domains also are
 Hajlasz-Sobolev $\dot {\bf M}^{1,p} $-extension domain with $p\ge 1$.
 Recall that for a given function $u $ in $\Omega$, we say $g$ is a  Haj\l asz gradient of $u$ (for short $g\in \mathbb D(u,\Omega)$) if
 $$|u(x)-u(y)|\le |x-y|[g(x)+g(y)]\quad\mbox{for almost all $(x,y)\in \Omega\times\Omega$}.$$
 The Haj\l asz Sobolev space
 $\dot {\bf M}^{1,p}(\Omega) $ is the set of all functions $u$ in $\Omega$  with
 $$\|u\|_{\dot {\bf M}^{1,p}(\Omega)}:=\inf_{g\in \mathbb D(u,\Omega)}\|g\|_{L^p(\Omega)}<\fz$$
Conversely, Haj\l asz  \cite{hkt08,hkt08b} essentially proved that Hajlasz-Sobolev $\dot {\bf M}^{1,p} $-extension domains must be Ahlfors $n$-regular; and by \cite{z14},
similar results hold true for  fractional Sobolev $\dot {\bf W}^{s,p} $-extension   for any $s\in(0,1)$ and  $p\ge1$. Note that $\dot {\bf W}^{n/p,p}(\Omega) =  \dot {\bf B}^{n/p}_{p,p}(\Omega)=\bf{\dot B}^\phi(\Omega) $ for any $p>n$ and $\phi(t)=t^p$.

 To prove  Theorem 1.1 (i), it suffices to  define a suitable linear extension operator and prove its boundedness.
 Following Jones \cite{j81}, to define the extension operator we have to find suitable reflecting cubes
of  Whitney cubes for $\rn\setminus\overline \Omega$.
If we use the reflecting cubes the same as in \cite{k98,hkt08,wxz,z14} which may have unbounded overlaps,
we cannot prove the boundedness of the extension operator  in general
 since the Young function may grows exponentially at $\fz$. See Remark 4.1 for details.
Instead, we use the reflecting cubes of Shvartsman \cite{s06,s07}, which have bounded overlaps (see Lemma 2.2), to define extension operator.
The bounded overlaps of reflecting cubes allow us to use the convexity of $\phi$, and also avoid using maximal functions. With some careful analysis, we finally obtain   the boundedness of extension operator.

 Theorem 1.1 (ii)  is proved  in section 5 by borrowing some ideas from \cite{hkt08,z14}.
 Precisely, we first prove
  $\Omega $  supports the following imbedding:
 there exists  positive constants $C_I(\Omega)$ and $C (n) $ such that
\begin{eqnarray}\label{im}
\inf \limits _{c \in R} \int_{B \cap \Omega} \exp\left( \frac{|u- c|}{\alpha} \right)\,dx \leq C (n) |B|\quad \mbox{for any ball   $B \subset \rn$.}
\end{eqnarray}
whenever  $u \in {\dot {\bf B}}^\phi (\Omega)$ and $\alpha  > C_{I}(\Omega)\|u\|_{{\dot {\bf B}}^\phi(  \Omega)}>0$. Then we calculate the precise $\|u\|_{{\dot {\bf B}}^\phi(  \Omega)}$-norm
of some cut-off functions. Using this  and the sub-exponential growth of $\phi$
following the idea   from \cite{hkt08} (see also \cite{hkt08b,z14}), we are able to prove $\Omega$ is Ahlfors $n$-regular.

As a byproduct, we have the following result.
\begin{cor} \label{c1.2}
Suppose that   $\phi$ is a Young function satisfying   \eqref{delta0} and   \eqref{subexp}.
Let $\Omega\subset\rn$ be any domain. The  following are equivalent:
\begin{enumerate}
 \item[(i)] $\Omega$ is Ahlfors $n$-regular;

  \item[(ii)]  $\Omega$ is a ${\dot {\bf B}}^\phi$-extension domain;

 \item[(iii)]  $\Omega$ supports the  imbedding \eqref{im}.
 \end{enumerate}
\end{cor}

\begin{rem}\rm
 We conjecture that  Theorem \ref{t1.1} (ii) holds without the additional assumption \eqref{subexp}.
  The difficult to remove \eqref{subexp} is to
   find a suitable imbedding properties of  $ {\dot {\bf B}}^\phi(  \rn) $ better than \eqref{im} when $\phi$ does not satisfies \eqref{subexp}.

 Note that   \eqref{im} is  always true when $\Omega$ is a $  \dot {\bf B} ^{\phi }$-extension domain,
 but it is not enough to prove that $\Omega$ is Ahlfors $n$-regular in general.
 If  $\phi(t)=e^{ t^\alpha }-\sum_{j=0}^{[n/\alpha] } t ^{\gz j}/j!$ for $t\ge0$,
by Lemma 2.5,  $\Omega$ supports the imbedding
 $$\fint_{B} \exp\left(\frac{ |u(x)- u_{B}|}{\alpha}\right)^\gz\,dx\le C(n) $$
whenever $\alpha>C(\gz,n)\|u\|_{ \dot {\bf B} ^{\phi_\gz}(\rn) }$.
However, when $  \dot {\bf B} ^{\phi_\gz}$-extension domain, such a imbedding is also not enough to prove  $\Omega$ is Ahlfors $n$-regular.
 \end{rem}

Notation used in the following is standard. The constant $C(n,\alpha, \phi)$ would vary from line to line and is independent of parameters depending only on $n, \alpha, \phi$. Constants with subscripts would not change in different occurrences , like $C_\phi$.
Given a domain, set $B_{\Omega}(x,r)=B(x,r)\cap \Omega$ for convenience.
We denote by $u_X$ the average of $u$ on $X$, namely, $u_X = \fint _{X} u \equiv \frac{1}{|X|} \int _{ X} u \, dx$. For a domain $\Omega$ and $x\in \rn$, we use $d(x, \Omega)$ to describe the distance from $x$ to $\Omega$.

\section{Some basic properties}

We list several basic properties  of Orlicz-Besov spaces.

\begin{lem}\label{12.4}
Suppose that $\phi$ is a Young function.  Let $\Omega \subset \rn$ be any domain.
Then   $ {\dot {\bf B}}^\phi (\Omega) \subset L^1(B\cap \Omega)$ for any ball $B\subset\rn$,
in particular, ${\dot {\bf B}}^\phi (\Omega) \subset    L^1(\Omega)$ when $\Omega$ is bounded.
\end{lem}
\begin{proof}

For any  $\alpha>\|u\|_{{\dot {\bf B}}^\phi(\Omega)}$, we have
$$\int_\Omega\int_\Omega \phi \left(\frac{|u(x)-u(y)|}{\alpha}\right)\,\frac{dydx}{|x-y|^{2n}}\le1.$$
By Fubini's theorem,  for almost all $x\in   \Omega$ we have
$$\int_\Omega \phi \left(\frac{|u(x)-u(y)|}{\alpha}\right)\,\frac{ dy}{|x-y|^{2n}}<\fz.$$
Fix such a point $x$.  For any $B=B(z,r)\subset \rn$ with $B\cap \Omega\ne\emptyset$,
we have $|x-y|\le |x|+|z|+r$ for all $y\in B\cap\Omega$, and hence
$$ \fint_{B\cap \Omega}  \phi \left(\frac{|u(x)-u(y)|}{\alpha}\right)\, dy  <\fz.$$
By Jessen's inequality, we have
 $$ \phi\left (\fint_{B\cap \Omega}     \frac{|u(x)-u(y)|}{\alpha}\, dy \right) <\fz,$$
 which implies that $$\fint_{B\cap \Omega}    |u(y) | \, dy \le |u(x)|+\fint_{B\cap \Omega}      |u(x)-u(y)| \, dy <\fz,$$
 that is, $u\in L^1 ({B\cap \Omega} )$ as desired.
\end{proof}

\begin{lem}\label{l3.3} Suppose that $\phi$ is a Young  function satisfying \eqref{delta0}.  Let $\Omega\subset\rn$ be any domain. Then $C_c^1(\Omega)\subset \dot{\bf B}^\phi(\Omega)$.
\end{lem}
\begin{proof}
Assume that  $L=\| u\|_{L^\fz(\Omega)} +\|Du\|_{L^\fz(\Omega)}>0$.
Let $V={\rm\,supp\,}u\Subset W\Subset\Omega$. Then
\begin{align*}
H:=\int_\Omega\int_\Omega\phi\left(\frac{|u (z)-u (w)|}\alpha\right)\frac{dzdw}{|z-w|^{2n}}
&\le  \int_W\int_{W}\phi\left(\frac{| z-w|}{\alpha/ L} \right)\frac{dzdw}{|z-w|^{2n}} + 2
\int_{\Omega\setminus W}\int_{V}\phi\left(\frac  {L}{\alpha } \right)\frac{dzdw}{|z-w|^{2n}}.
\end{align*}
By \eqref{delta0}, we have
\begin{align*}
 \int_W\int_W\phi\left(\frac{| z-w|}{\alpha/L} \right)\frac{dzdw}{|z-w|^{2n}}&\le \int_V\int_{B(w,2|\diam W|)}\phi\left(\frac{| z-w|}{\alpha/L} \right)\frac{dz}{|z-w|^{2n}}dw\\
&=n\omega_n\int_W\int_0^{2|\diam W|}\phi\left(\frac{t}{\alpha/L} \right)\frac{ dt}{t^{n+1}}\,dw\\
&=n\omega_n|W| (\frac{L} \alpha)^n \int_0^{2L|\diam W|/\alpha}\phi\left(s\right)\frac{ ds}{s^{n+1}} \\
&=n\omega_n |W| 2^{-n}|\diam W|^{-n}\phi\left(\frac{2L|\diam W|}{\alpha } \right).
\end{align*}
Moreover,
\begin{align*}
2\int_{\Omega\setminus W}\int_{V}\phi\left(\frac  1{\alpha/L } \right)\frac{dzdw}{|z-w|^{2n}}
&\le 2\phi\left(\frac  L{\alpha  } \right)\int_{V}\int_{\Omega\setminus B(z,\dist(V,W^\complement) )}\frac{dwdz}{|z-w|^{2n}}
 \le 2\omega_n \phi\left(\frac  {L}{\alpha }\right)|V|\dist(V,W^\complement)^{-n}.
\end{align*}
Obviously, letting $\alpha$ sufficiently enough and using the convexity of $\phi$, we have $H\le1$.  That is, $u \in \dot{\bf B}^\phi(\Omega)$ as desired.
\end{proof}

 The following  Poincar\'e type inequality is needed in Section 4. Below denote by $\omega_n$ the area of  the unit sphere $S^{n-1}$.
\begin{lem}\label{l2.5} Suppose that $\phi$ is a Young  function.
For any ball $B\subset\rn$ and $u\in {\dot {\bf B}}^\phi(B)$, we have
$$\fint_{B}\phi \left(\frac{ |u(x)- u_{B}|}{\alpha}\right)\,dx\le \omega_n^2$$
when $\alpha>\|u\|_{{\dot {\bf B}}^\phi(B)}$,
and
$$\fint_{B} |u(x)- u_{B}|\,dx\le  \phi^{-1}(\omega_n^2)\|u\|_{{\dot {\bf B}}^\phi(B)}.$$
\end{lem}
\begin{proof}

Let  $ u \in {\dot {\bf B}}^\phi(B)$. For any   $\alpha > \|u\|_{{\dot {\bf B}}^\phi(B)}$, by Jensen's inequality,   we have
\begin{eqnarray*}
\phi \left(\frac{ \fint_{B} |u(x)- u_{B}|\,dx}{\alpha}\right)&&\le \fint_{B}\phi \left(\frac{ |u(x)- u_{B}|}{\alpha}\right)\,dx\\
&&\le  \fint_B \fint_B \phi \left(\frac{|u(x)-u(y)|}{\alpha}\right)\,dydx\\
&& \le \omega_n^2\int_B \int_B \phi \left(\frac{|u(x)-u(y)|}{\alpha}\right)\,\frac{dydx}{|x-y|^{2n}}\le \omega_n^2,
\end{eqnarray*}
that is,
$$\fint_{B} |u(x)- u_{B}|\,dx\le \alpha \phi^{-1}(\omega_n^2).$$
Letting $\alpha\to \|u\|_{{\dot {\bf B}}^\phi(B)}$, we obtain
$$\fint_{B} |u(x)- u_{B}|\,dx\le  \phi^{-1}(\omega_n^2)\|u\|_{{\dot {\bf B}}^\phi(B)}$$ as desired.
\end{proof}

As a consequence of Lemma \ref{l2.5}, we have the following imbedding.
Denote by   $BMO(\Omega)$  the space of functions with bounded mean oscillations, that is,
  the collection of $u\in L^1_\loc(\Omega)$ such that
$$\|u\|_{BMO (\Omega)}=\sup_{B\subset \Omega} \fint_{B} |u(x)- u_{B}|\,dx<\infty.$$

\begin{cor}\label{c2.x} Suppose that $\phi$ is a Young  function.  Let $\Omega\subset\rn$ be any domain.
We have $ {\dot {\bf B}}^\phi (\Omega) \subset BMO(\Omega)$ and $\|u\|_{BMO(\Omega)}\le \phi^{-1}(\omega_n^2)\|u\|_{{\dot {\bf B}}^\phi(\Omega)}$ for all $u\in  {\dot {\bf B}}^\phi (\Omega)$.
\end{cor}

Note that if $$\phi_\gz(t)= e^{ t^\gz}-\sum_{j=0}^{[n/\gz]} \frac{ t ^{ \gz j }}{j!}$$ with $\gz\ge1$, then
$ \phi_\gz$ a Young's  function satisfying \eqref{delta0}. Denote by ${\dot {\bf B}}^{\phi_\gz}(\rn)$
the associated Orlicz-Besov space.
\begin{lem}\label{l2.xx5}  For $\gz\ge1$, there exists constant $C(\gz,n)\ge 1$ such that
for any $u\in {\dot {\bf B}}^{\phi_\gz}(\rn)$ and ball $B\subset\rn$, we have
$$\fint_{B} \exp\left(\frac{ |u(x)- u_{B}|}{\alpha}\right)^\gz\,dx\le C(n) $$
whenever $\alpha>C(\gz,n)\|u\|_{ \dot {\bf B} ^{\phi_\gz}(\rn) }$.
\end{lem}

\begin{proof} 
By Corollary \ref{c2.x}, we have $u\in BMO(\rn)$ and
$\|u\|_{BMO(\rn)}\le \phi^{-1}(\omega_n^2)\|u\|_{{\dot {\bf B}}^\phi(\rn)}$.
Thus  by the John-Nirenberg inequality,  we have
$$\fint_{B}  |u(x)- u_{B}| ^{[n/\gz]}\,dx\le  C(\gz,n)\|u\|^{[n/\gz]}_{BMO(\rn)}
\le C(\gz,n)\|u\|^{[n/\gz]}_{{\dot {\bf B}}^\phi(\rn)}.$$
Thus for all $1\le j\le [n/\gz]$, we have
$$\fint_{B} \left(\frac{|u(x)- u_{B}|}{\alpha}\right)^{\gz j}\,dx\le 1/n $$
when
$\alpha\ge C(\gz,n) \|u\|^q_{{\dot {\bf B}}^\phi(\rn)}$ for some constant $C(\gz,n)$. Note that by Lemma 2.3, one has
$$\fint_{B}\phi_\gz \left(\frac{ |u(x)- u_{B}|}{\alpha}\right)\,dx\le1/n$$
when $\alpha>n\omega_n^2\|u\|_{{\dot {\bf B}}^\phi(\rn)}$.
Since $$ e^{ t^\gz}=\phi_\gz(t) +1+\sum_{j=1}^{[n/\gz]} \frac{ t ^{ \gz j }}{j!}$$
 we obtain
$$\fint_{B} \exp\left(\frac{ |u(x)- u_{B}|}{ \alpha}\right)^\gz\,dx\le  3.$$
when $\alpha>[n\omega_n^2+C(\gz,n)]\|u\|_{{\dot {\bf B}}^\phi(\rn)}$.
\end{proof}

\section{Whitney's decomposition  and the reflected quasi-cubes\label{s2}}

In this section, we always let $\Omega$ be an Ahlfors $n$-regular domain.
 Observe that $|\partial\boz|=0$; see \cite[Lemma 2.1]{s06} and also \cite{z14,hkt08}.
Moreover, $\diam\Omega=\infty$ if and only if $|\Omega|
=\infty.$ 
Write $U:=\rn\setminus\overline \Omega $. Without loss of generality, we assume $U  \neq \emptyset$.
It's well know that $U$ admits a Whitney decomposition.

\begin{lem}
There exists a collection ${\mathscr W}=\{Q_i\}_{i \in \nn}$ of   (closed) cubes  satisfying
\begin{enumerate}
\item[ (i)] $U = \cup_{i \in \nn}Q_i$, and  $ Q^\circ_{k}  \cap Q^\circ_{i}  =\emptyset $ for all $i, k \in \nn$ with $i\neq k$;

\item[ (ii)] $ l(Q_k) \leq \dist (Q_k , \partial \Omega) \leq 4\sqrt{n} l(Q_k)$;

\item[ (iii)] $\frac{1}{4} l(Q_k) \leq l(Q_i) \leq 4l(Q_k)$ whenever $Q_k \cap Q_i \neq \emptyset$.
\end{enumerate}
\end{lem}

The following basic properties of Whitney's decomposition are used quite often in Section 4.
For any $Q\in \mathscr W$, denote by $N (Q)$ the neighbor cubes of $Q$ in $\mathscr W$, that is,
$$N (Q):=\{P\in \mathscr W, P \cap Q \neq \emptyset\}.$$
 Then, by (iii) there exists an integer $\gamma_0$ depending only on $n$ such that
\begin{eqnarray} \label{e2.w1}
\sharp N(Q) \leq \gamma_0\quad\mbox{for all $Q \in \mathscr W  $.}
\end{eqnarray}
By (iii) again, for any $P,Q\in \mathscr W$ we know that
\begin{eqnarray} \label{e2.xx1}\mbox{$P\in N(Q)$ if and only if $Q\in N(P)$,  if and only if $\frac98P\cap \frac98Q\ne\emptyset$}.\end{eqnarray}
It then follows that
\begin{eqnarray} \label{e2.w2}
\frac1{|Q|}\int_U  \chi_{\frac{9}{8}Q}(x) \,dx\le 4^n\gamma_0\quad\mbox{for all $Q \in \mathscr  W  $.} \end{eqnarray}
Indeed, by \eqref{e2.xx1} we write
$$ \frac1{|Q|}\int_U  \chi_{\frac{9}{8}Q}(x) \,dx=
 \sum_{P\in N(Q)}\frac1{|Q|}\int_P  \chi_{\frac{9}{8}Q}(x) \,dx.$$
By $l_Q\le 4l_P$ given in (iii), and \eqref{e2.w1},
 we arrive at
$$\frac1{|Q|}\int_U  \chi_{\frac{9}{8}Q}(x) \,dx\le \sum_{P\in N(Q)} \frac{|P|}{|Q|} \le 4^n\gamma_0 $$
as desired.

Below we recall the reflected quasi-cubes of Whitney's cubes as given by Shvartsman \cite[Theorem 2.4]{s06}.
For any $\epsilon>0$, set
$$
\mathscr W_\ez :=\{Q  \in \mathscr W : l_Q <  \frac{1}{\epsilon}\diam \Omega \}.
$$
Obviously,  $\mathscr W= \mathscr W_\ez $ for all $\ez>0$ if $\diam\Omega=\fz$, and $\mathscr W_\ez \subsetneq \mathscr W$ for any $\ez>0$ if $\diam\Omega<\fz$.

For any $Q = Q( x_Q ,l_Q  )\in \mathscr W_\ez $, fix any $ x_Q ^\ast \in \Omega$  so that
$\dist(Q,\Omega)=\dist(x,Q)$.
By  Lemma 3.1 (ii), one has $$\widetilde Q^\ast:=Q(  x^\ast_Q , l_Q ) \subset 10\sqrt{n} Q.$$
Set
$$
 \widetilde Q^{\ast\ez} := (\ez Q^\ast  \cap \Omega)\setminus   \left( \bigcup  \{{\epsilon}P^\ast: P \in \mathcal{A}^\ez_Q\}   \right),
$$
where $$
\mathcal{A}^\ez _{Q } := \left\{P  \in \mathscr W_\ez  : \ez  \wz P^\ast \cap \ez  \wz Q^\ast  \neq \emptyset, l_P \leq  \ez l_{Q}\right\}.
$$
Below, when $\ez $ is small enough, we define $\widetilde Q^{\ast\ez}$ as
  reflected quasi-cubes  of $Q\in \mathscr W_\ez $ so that they enjoy some nice properties; see \cite[Theorem 2.4]{s06} for the proof, here we omit the details.

\begin{lem}\label{l2.1} Let $\ez_0 = [ C_A(\Omega) /2\gz_0]^{1/n}/ (30\sqrt{n})  $.
Denote by $Q^\ast= \widetilde Q^{\ast\ez_0}$ as quasi-cubes  of any cube $Q\in \mathscr W_{\ez_0}$.
Then the following hold:
\begin{enumerate}
\item[  (i) ] $ Q^\ast\subset (10\sqrt{n}Q) \cap \Omega$ for any $Q\in \mathscr W_{\ez_0}$;

\item[  (ii) ] $|Q| \leq \gamma_1 |Q^\ast|$ whenever $Q\in \mathscr W_{\ez_0}$;
\item[  (iii) ] $\sum \limits _{Q \in \mathscr W_{\ez_0}} \chi_{ Q^\ast} \leq \gamma_2$.
\end{enumerate}
\noindent
Above $\gamma_1$ and $\gamma_2$ are positive constants depending only on $n$ and $ C_A(\Omega) $.
\end{lem}

If $\Omega$ is bounded, we let $Q^\ast=\Omega$ as the reflected  quasi-cube  of  any cube $Q\in\mathscr W\setminus \mathscr W_{\ez_0}\ne\emptyset$.
Write $$\mathscr W_{\ez_0}^{(k)}=\{Q\in N(P): P\in \mathscr W_{\ez_0}^{(k-1)}\}\quad\forall k\ge1,$$
where $\mathscr W_{\ez_0}^{^{(0)}}=  \mathscr W_{\ez_0} $. That is, $\mathscr W_{\ez_0}^{(k)}$ is the $k^{\rm th}$-neighbors of $ \mathscr W_{\ez_0} $.
\begin{align} \label{vk.1}
V^{(k)}:= \bigcup \{x\in Q; Q\in    \mathscr W_{\ez_0}^{(k)}\}\quad\forall k\ge0.
\end{align}

Since  $Q^\ast=\Omega $ for $Q\notin \mathscr W_{\ez_0}$, by  Lemma 3.3 (iii) we have
 $$\sum\limits_{Q \in\mathscr W^{^{(k)}}_{\ez_0}}\chi_{Q^\ast}\le  \sum\limits_{Q \in\mathscr W  _{\ez_0}}\chi_{Q^\ast}+ \sharp(\mathscr W^{^{(k)}}_{\ez_0} \setminus \mathscr W _{\ez_0})\chi_{\Omega}
\le [\gamma_2 +\sharp(\mathscr W_{\ez_0}^{(k)}\setminus  \mathscr W_{\ez_0}  )]\chi_{\Omega}\quad\forall k\ge1.$$
For $Q\in \mathscr W^{^{(k)}}_{\ez_0} \setminus \mathscr W _{\ez_0}$, observe that $l_Q\ge \frac1{\ez_0}\diam\Omega$ and
$l_Q\le 4^k l_{P}\le \frac{4^k}{\ez_0}  \diam\Omega$ for some $P\in \mathscr W_{\ez_0}$.
Thus, by Lemma 3.1 (ii), we have  $$Q\subset Q(\bar x, \diam\Omega+8\sqrt n\frac{4^k}{\ez_0}\diam\Omega)$$ for any fixed $\bar x\in \Omega$, and hence
$$\sharp(\mathscr W_{\ez_0}^{(k)} \setminus \mathscr W_{\ez_0})\le ( 1+ 8\sqrt n\frac{4^{k  } }{\ez_0})^n\ez_0^n \le (\ez_0+4^{k+2}\sqrt{n})^n.$$
This yields that
\begin{equation}\label{eq2.xx2}\sum\limits_{Q \in \mathscr W_{\ez_0}^{(k)}}\chi_{Q^\ast}   \le \gamma_2+(\ez_0+4^{k+2}\sqrt{n})^n\quad\forall k\ge1.
\end{equation}

Finally, associated to $\mathscr W$,  one has the following partition of unit of $U$.
\begin{lem}
There exists a family  $\{\vz_Q:Q\in \mathscr W\}$ of functions such that
\begin{enumerate}
\item[(i)] for each $Q  \in \mathscr  W$,   $0\le \varphi_Q  \in C^{\infty}_0(\frac{17}{16}Q)$;
 \item[(ii)]for each $Q  \in \mathscr W$, $|\nabla \vz_Q|\le L/l_Q$;
  \item[(iii)]  $\sum_{Q  \in \mathscr W }  \vz =\chi_{U} $ .
\end{enumerate}
\end{lem}

\section{Proof of    Theorem \ref{t1.1}(i) \label{s3}}

 It suffices to prove the existence of a bounded linear operator $E: {\dot {\bf B}}^\phi(\Omega)\to {\dot {\bf B}}^\phi(\rn)$ such that $Eu|_\boz=u$ for all $u\in {\dot {\bf B}}^\phi(\Omega)$.
Define the  operator $E$  by
\begin{equation*}
Eu(x)\equiv\lf\{\begin{array}{ll}
u(x),&x\in \boz,\\
0&x\in\partial \boz,\\
 \sum\limits_{Q  \in \mathscr W  }  \varphi_Q(x) u_{Q^\ast} , \quad &x\in  U
\end{array}\r.
\end{equation*}
 for any $u \in {\dot {\bf B}}^\phi(\Omega)$.
 Recall that $\mathscr W$ is the Whitney cubes of $U$ as  in Lemma 3.1 and $\{\vz_Q\}_{Q\in W}$  as  in Lemma 3.3;
   that $ {Q^\ast}$ is the reflected quasi-cube of $Q\in \mathscr W_{\ez_0}$ as given in Lemma 3.2,
 and $Q^\ast=\Omega$ if $Q\in \mathscr W\setminus \mathscr W_{\ez_0}$ (when $\Omega$ is bounded).
 By Lemma 2.1, $u_{Q^\ast}=\frac1{|Q^\ast|}\int_{Q^\ast}u\,dx$ is always finite.

Obviously, $E$ is linear, $Eu|_{\Omega}= u $ in $\Omega$, and moreover, if $\|u\|_{{\dot {\bf B}}^{\phi}(\Omega)}=0$, then   $u$ and hence $Eu$ must be a  constant function essentially.
Thus,  to prove the boundedness of $E:{\dot {\bf B}}^\phi(\Omega) \to {\dot {\bf B}}^\phi(\rn)$,
by   the definition of the norm
$\|\cdot\|_{{\dot {\bf B}}^{\phi}(\rn)}$, we  only need to find a
 constant $M >0$ depending only on $n$, $C_A(\Omega)$ and $\phi$  such that
\begin{eqnarray}\label{eq3.y1}
H(\alpha) :=\int_{\rn} \int_{\rn} \phi \left(\frac{|Eu(x)-Eu(y)|}{\alpha }\right) \, \frac{dydx}{|x-y|^{2n}}  \leq 1.
\end{eqnarray}
whenever $\|u\|_{{\dot {\bf B}}^{\phi}(\Omega)}=1$ and $\alpha>M$.
  Below we assume that $\|u\|_{{\dot {\bf B}}^{\phi}(\Omega)}=1$
 Since $|\partial\Omega|=0$, one writes
\begin{eqnarray*}
H(\alpha) && = \int_{\Omega} \int_{\Omega} \phi \left(\frac{|u(x)-u(y)|}{\alpha }\right) \, \frac{dydx}{|x-y|^{2n}}+
2\int_{U} \int_{\Omega} \phi \left(\frac{|Eu(x)-u(y)|}{\alpha }\right) \, \frac{dydx}{|x-y|^{2n}}\\
&&\quad +\int_{U} \int_{U} \phi \left(\frac{|Eu(x)-Eu(y)|}{\alpha}\right) \, \frac{dydx}{|x-y|^{2n}}\\
&&  =:H_1(\alpha) + 2H_2(\alpha)+H_3(\alpha).
\end{eqnarray*}
To get \eqref{eq3.y1}, it  suffices to find constants $  M_i \ge 1$ depending only on $n$, $C_A(\Omega)$ and $\phi$ such that $H_i(\alpha) \le 1/4$   whenever $\alpha\ge  M_i $ for $ i=1,2,3$.
Indeed, by taking $M=M_1+M_2+M_3$, we have $H(\alpha)\le1$ whenever $\alpha\ge M$.

Firstly, we may let $ M_ 1=4$.  Indeed, if $\alpha> 4$ that is, $\alpha/4>1$,
by the convexity of $\phi$ and $\|u\|_{{\dot {\bf B}}^{\phi}(\Omega)}=1$, we have
$$H_1(\alpha)\le \frac14\int_{\Omega} \int_{\Omega} \phi \left(\frac{|u(x)-u(y)|}{\alpha/4 }\right) \, \frac{dydx}{|x-y|^{2n}}\le \frac14.$$

To find $ M_ 2$ and $ M_ 3$, we consider two cases: $\diam\Omega=\fz$ and $\diam\Omega<\fz$.

\medskip
\noindent
{\it Case $\diam\Omega=\fz$.}
To find $ M_ 2$, for any $x\in U $  and  $y\in\boz$,
since
 $\sum \limits _{Q \in \mathscr W } \varphi_Q (x) =1$ by Lemma 3.3,  one has
$$
Eu(x)-u(y)=\sum\limits_{Q  \in  \mathscr W   }\varphi_Q (x)[u_{Q^\ast}-u(y)],$$
 and hence,
by the convexity of $\phi$ and Jensen's inequality,
\begin{eqnarray*}
\phi\left(\frac{|Eu(x)-u(y)|}{\alpha }\right)&&
\le \phi\left(\sum\limits_{Q  \in \mathscr  W   }\varphi_Q(x) \frac{|u_{Q^\ast}-u(y)|} {\alpha }\right)\\
 && \le   \sum\limits_{Q  \in \mathscr W  } \varphi_Q(x) \phi  \left(\fint_{Q^\ast} \frac{|u(z)-u(y)|}{\alpha }\,dz\right)\\
&& \le  \sum\limits_{Q  \in \mathscr  W  }  \varphi_Q(x)  \fint _{Q^\ast}\phi\left(\frac{|u(z)-u(y)|}{\alpha }\right) \,dz.
\end{eqnarray*}
If $
\varphi_Q(x)\ne 0$, then $x\in \frac{17}{16}Q$. For $z\in Q^\ast$, by $Q^\ast\subset 10\sqrt n Q$,
we have
 $  |x-z| \leq  20  n l(Q)$.
 Since  $|x-y| \geq d(x,\, \Omega) \geq  l(Q)$,
 we have  $|x-z|\le 20  n |x-y|$, that is , $$|y-z| \leq |x-y| + |x-z| \leq 21  n |x-y|$$
So we have
\begin{eqnarray*}
\int_\boz \phi\left(\frac{|Eu(x)-u(y)|}{\alpha}\right)  \, \frac{dy}{|x-y|^{2n}}
&& \le  (21  n)^{2n}\sum\limits_{Q \in \mathscr W}\varphi_Q(x) \fint _{Q^\ast} \int_\boz\phi\left(\frac{|u(z)-u(y)|}{\alpha}\right)  \, \frac{dzdy}{|z-y|^{2n}}.
\end{eqnarray*}
Thus, by Lemma 3.2 (ii) we write
 \begin{eqnarray*}
H_2(\alpha)
&& \leq  2(21  n)^{2n} \int_U  \sum\limits_{Q \in \mathscr W}\varphi_Q(x)  \fint _{Q^\ast} \int_{\Omega}\phi\left(\frac{|u(z)-u(y)|}{\alpha}\right)  \frac{dzdy}{|y-z|^{2n}} \,dx\\
&& \leq 2\gamma_1(21  n)^{2n}   \sum\limits_{Q \in \mathscr W}\left(\frac1{|Q|}\int_U\varphi_Q(x) \,dx\right) \int _{Q^\ast} \int_{\Omega}\phi\left(\frac{|u(z)-u(y)|}{\alpha}\right)  \frac{dzdy}{|y-z|^{2n}} .
\end{eqnarray*}
Since    $\varphi_Q  \le \chi_{\frac98Q} $  as given in Lemma 3.3,    by \eqref{e2.w2}
we have
$$\frac1{|Q|}\int_U  \vz_Q(x) \,dx\le \frac1{|Q|}\int_U  \chi_{\frac98Q}(x) \,dx  \le 4^n\gamma_0,$$
which implies that
 \begin{eqnarray}\label{eq4.w2}
H_2(\alpha)
&& \leq 2\gamma_1 4^n\gamma_0(21 n)^{2n}  \sum\limits_{Q \in \mathscr W} \int _{Q^\ast}  \int_{\Omega}\phi\left(\frac{|u(z)-u(y)|}{\alpha}\right)   \frac{dzdy}{|y-z|^{2n}} .
\end{eqnarray}
By $\sum_{Q\in \mathscr W}\chi_{Q^\ast}\le \gamma_2$ as in Lemma 3.2 (iii), we obtain
 \begin{eqnarray*}
H_2(\alpha)
&& \leq 2\gamma_1 4^n\gamma_0  \gamma_2 (21  n)^{2n}  \int _\Omega  \int_{\Omega}\phi\left(\frac{|u(z)-u(y)|}{\alpha}\right)  \frac{dzdy}{|y-z|^{2n}}.
\end{eqnarray*}
Take $ M_ 2= 8\gamma_1 4^n\gamma_0  \gamma_2 (21  n)^{2n}$.
By the convexity of $\phi$ again, if $\alpha >   M_ 2  $, we have $H_2(\alpha)\le 1/4$.

To find $ M_ 3$, for each $x \in U$   set
\begin{eqnarray*}
X_1(x) := \left\{ y \in U : |x-y| \geq \frac{1}{132  n} \max\{ d(x, \Omega) , d(y, \Omega)\}\right\}\quad{\rm and}\\
X_2(x) :=U\setminus X_1(x)= \left\{ y \in U : |x-y| < \frac{1}{132  n} \max\{ d(x, \Omega) , d(y, \Omega)\}\right\}.
\end{eqnarray*}
Write
\begin{eqnarray*}
H_3(\alpha) && = \int_{U} \int_{X_1(x)} \phi \left(\frac{|Eu(x)-Eu(y)|}{\alpha }\right) \, \frac{dydx}{|x-y|^{2n}} +\int_{U} \int_{X_2(x)} \phi \left(\frac{|Eu(x)-Eu(y)|}{\alpha }\right) \, \frac{dydx}{|x-y|^{2n}}\\
&& =H_{31}(\alpha) + H_{32}(\alpha)
\end{eqnarray*}

Below, we show that there exists $ M_ {3i}\ge1$ such that if $\alpha> M_ {3i}$, then $H_{3i}\le 1/8$ for $i=1,2$.
If this is true, then letting $ M_ 3:=\max\{ M_ {31}, M_ {32}\}$,
for $\alpha> M_ 3 $ we have $H_{3}\le \frac14$ as desired.

To find $ M_ {31}$, for $x\in U$ and $y\in X_1(x)$,  since $$\sum \limits _{Q  \in \mathscr  W  } \varphi_Q (x)=\sum \limits _{P  \in \mathscr W  } \varphi_P (y)=1,$$ we have
 \begin{eqnarray*}
Eu(x)-Eu(y)&&=\sum\limits_{P \in \mathscr W  }\sum\limits_{Q \in \mathscr W  }\varphi_Q(x)\varphi_P(y)[u_{Q^\ast}-u_{P^\ast}]\\
&&= \sum\limits_{P \in \mathscr W  }\sum\limits_{Q \in \mathscr W  }\varphi_Q(x)\varphi_P(y)\fint_{Q^\ast} \fint_{P^\ast }[u(z)-u(w)]\,dzdw.
\end{eqnarray*}
Applying the convexity of $\phi$ and  Jensen's inequality, one obtains
\begin{eqnarray*}
\phi\left(\frac{|Eu(x)-Eu(y|}{\alpha }\right)
 && \le   \sum\limits_{Q  \in\mathscr W  } \sum\limits_{P \in\mathscr W  }\varphi_Q(x)\varphi_{P}(y) \phi  \left(\fint_{Q^\ast} \fint_{P^\ast }\frac{|u(z)-u(w)|}{\alpha }\,dzdw\right)\\
&& \leq \sum \limits _{Q \in\mathscr W }\sum \limits _{P \in\mathscr W }  \varphi_Q(x)\varphi_P(y)
\fint_{Q^\ast} \fint_{P^\ast} \phi \left(\frac{|u(z)- u(w)|}{\alpha }\right)\, dwdz.
\end{eqnarray*}

For $x\in Q$ and $z\in Q^\ast$, by $Q^\ast\subset 10 \sqrt n Q$, we have $|x-z|\le 10  n  l_Q\le 10  n d(x,\Omega)$. Similarly, for $y\in  P$,
 and $ w \in P^\ast$, we have $|y-w|\le 10  n  d(y,\Omega)$. If $y\in X(x)$, that is,
 $132  n|x-y| \geq  \max\{d(x, \Omega), d(y, \Omega)\}$, we further have
 $$ |z-w|\le |x-z|+ |x-y| + |y-w| \le   2641n |x-y|.$$
Thus
\begin{eqnarray*}
H_{31}(\alpha)&& \le (2641 n )^{2n}
\int_{U} \int_{X_1(x)} \sum \limits _{Q \in \mathscr W }\sum \limits _{P \in\mathscr W }  \varphi_Q(x)\varphi_P(y)
\fint_{Q^\ast} \fint_{P^\ast} \phi \left(\frac{|u(z)- u(w)|}{\alpha }\right)\, \frac{dwdz} {|z-w|^{2n}} \, {dydx}
\end{eqnarray*}
By  $|Q|\le \gamma_1|Q^\ast|$ and $|P|\le \gamma_1|P^\ast|$ as given in Lemma 2.2 (ii), we   have
  \begin{eqnarray*}
H_{31}(\alpha)
&&\le (2641n )^{2n}\gz_1^{2}\sum\limits _{Q \in \mathscr W }\sum \limits _{P \in\mathscr W }
\left(\frac1{|Q|}\int_{U}\varphi_Q(x)\,dx \frac1{|P|}\int_{U} \varphi_P(y)\, dy\right)
\int_{Q^\ast} \int_{P^\ast} \phi \left(\frac{|u(z)- u(w)|}{\alpha }\right)\, \frac{dwdz} {|z-w|^{2n}}.
\end{eqnarray*}
By Lemma 3.3 and \eqref{e2.w2} we have
$$\frac1{|Q|}\int_{U}\varphi_Q(x)\,dx \frac1{|P|}\int_{U} \varphi_P(y)\, dy\le( 4^n\gamma_0)^2.$$
 Thus
\begin{eqnarray*}
H_{31}(\alpha)
&&\le (2641 n )^{2n}\gz_1^{2}  (4^n\gamma_0)^2 \sum\limits _{Q \in \mathscr W }\sum \limits _{P \in\mathscr W }
\int_{Q^\ast} \int_{P^\ast} \phi \left(\frac{|u(z)- u(w)|}{\alpha }\right)\, \frac{dwdz} {|z-w|^{2n}}.
\end{eqnarray*}
Observing $\sum\limits _{Q \in \mathscr W }\chi_{Q^\ast}\le \gz_2$ as in Lemma 3.2 (iii), we  arrive at
\begin{eqnarray*}H_{31}(\alpha)
&&\le (2641n )^{2n}\gz_1^{2} \gz_2^2 (4^n\gamma_0)^2
\int_{U} \int_{U} \phi \left(\frac{|u(z)- u(w)|}{\alpha }\right)\, \frac{dwdz} {|z-w|^{2n}}.
\end{eqnarray*}
Letting $ M_ {31}=8   (2641 n )^{2n}\gz_1^{2} \gz_2^2 (4^n\gamma_0)^2$.
If $\alpha>  M_ {31} $, by the convexity of $\phi$ again
we have  $H_{31}(\alpha)\le 1/8$.

To find $ M_ {32}$, write
\begin{eqnarray*}
H_{32}(\alpha) && =  \int_{U}  \sum_{P\in\mathscr W}\int_{P\cap X_2(x)} \phi \left(\frac{|Eu(x)-Eu(y)|}{\alpha }\right) \, \frac{dy}{|x-y|^{2n}} dx.
\end{eqnarray*}

Let $x\in U$ and $y\in X_2(x) \cap P$ for some $P\in \mathscr W$.
Since  $$\sum \limits _{Q  \in \mathscr W }\left[ \varphi_Q (x) - \varphi_Q (y)\right]= 0,$$ we write
  $$ Eu(x)-Eu(y)=\sum \limits _{Q  \in \mathscr W }\left[ \varphi_Q (x) - \varphi_Q (y)\right] u_{Q^\ast} = \sum \limits _{Q  \in \mathscr W }\left[ \varphi_Q (x) - \varphi_Q (y)\right][u_{Q^\ast}-u_{P^\ast }].$$
Note that by  Lemma 3.3, $$|\nabla \vz_Q|\le \frac {L}{l_Q}\chi_{\frac{17}{16}Q} .$$
One gets
   $$ |Eu(x)-Eu(y)|\le L\sum \limits _{Q  \in\mathscr W } \frac{|x-y|}{l_Q}\left[\chi_{\frac{17}{16}Q}(x)+\chi_{\frac{17}{16}Q}(y)\right]|u_{Q^\ast}-u_{P^\ast }|.$$
Moreover, we have
\begin{equation}\label{eq4.x1} |Eu(x)-Eu(y)|\le 2L\sum \limits _{Q  \in N(P) } \frac{|x-y|}{l_Q} \chi_{\frac{9}{8}Q}(x) |u_{Q^\ast}-u_{P^\ast }|.
\end{equation}
Indeed, 
since $y\in X_2(x)$,  that is, $|x-y|\le \frac1{132  n} \max\{d(x,\Omega),d(y,\Omega)\} $,
taking $\bar{y} \in \bar{\Omega}$ with $|y-\bar{y}|=d(y,\Omega)$   we have
\begin{eqnarray*}
d(x,\Omega)\leq |x-\bar{y}| \leq |x-y|+|y-\bar{y}|
 \leq \frac{1}{132  n }d(x,\Omega)+ \frac{1+132  n}{132  n}d(y,\Omega),\nonumber
\end{eqnarray*}
which implies $$d(x,\Omega)\leq \frac{132  n+1}{132  n-1}d(y,\Omega),$$
Similarly,  we have $$d(y,\Omega)\leq \frac{132  n+1}{132 n-1}d(x,\Omega).$$ 
Thus, $$|x-y|\le \frac1{132  n}\frac{132  n+1}{132  n-1}d(x,\Omega).$$

If $y\in  \frac{17}{16} Q$, by $y\in P$ we have $Q\in N(P)$, and hence
$$   d(y,\Omega)\le d(y,Q)+\max_{z\in Q}d(z,\Omega) \leq \frac{1}{16} \sqrt n l_{  Q} + 4 \sqrt nl_{  Q} \le  \frac{65}{16}  \sqrt nl_{  Q},$$
Thus $$|x-y|\le \frac1{132  n}\frac{132  n+1}{132  n-1}\times \frac{65}{16}  \sqrt nl_{  Q}\le \frac{1}{32\sqrt n} l_{  Q},$$
which implies that $x \in \frac{9}{8}Q$. 
Moreover, if $x\in \frac{17}{16}Q$, similarly we have $y\in \frac{9}8 Q$, and hence $Q\in N(P)$.
We   conclude that
$$
  \chi_{\frac{17}{16}Q}(x)+\chi_{\frac{17}{16}Q}(y)=0 $$
when $Q\notin N(P)$, and
 $$
  \chi_{\frac{17}{16}Q}(x)+\chi_{\frac{17}{16}Q}(y) \le 2 \chi_{\frac{9}{8}Q}(x)$$
  when $Q\in N(P)$,
 This gives  \eqref{eq4.x1}.

Note that by Lemma 3.1(iii), $\sum \limits _{Q  \in \mathscr W }  \chi_{\frac{17}{16}Q}(x) \le \gamma_0$.
From  the convexity of $\phi$ and \eqref{eq4.x1}  it follows that
\begin{eqnarray*}
\phi \left(\frac{|Eu(x)-Eu(y)|}{\alpha }\right)&&\le  \phi\left (\sum \limits _{Q  \in N(P) } \frac{|x-y|}{l_Q} 2\chi_{\frac{9}{8}Q}(x) \frac{|u_{Q^\ast}-u_{P^\ast }|}{ \alpha/L}\right)\\
&& \le \frac1{  \gamma_0}\sum \limits _{Q  \in N(P) }  \chi_{\frac{9}{8}Q}(x) \phi\left(\frac{|x-y|}{l_Q}\frac{|u_{Q^\ast}-u_{P^\ast }|}{\alpha/ 2L \gamma_0}\right).
\end{eqnarray*}

Therefore,  we obtain
  \begin{eqnarray*}
H_{32}(\alpha) &&\le   \frac1{  \gamma_0}\int_{U} \sum_{P\in \mathscr W}\int_{P\cap X_2(x)}  \sum \limits _{Q  \in N(P) }  \chi_{\frac{9}{8}Q}(x) \phi\left(\frac{|x-y|}{l_Q}\frac{|u_{Q^\ast}-u_{P^\ast }|}{\alpha/ 2L \gamma_0}\right)\, \frac{dydx}{|x-y|^{2n}}\\
&&=   \frac1{  \gamma_0}\int_{U}\sum_{P\in\mathscr W} \sum \limits _{Q  \in N(P)}  \chi_{\frac{9}{8}Q}(x)   \int_{P\cap X_2(x)}\phi\left(\frac{|x-y|}{l_Q}\frac{|u_{Q^\ast}-u_{P^\ast }|}{\alpha/ 2L \gamma_0}\right)\, \frac{dy}{|x-y|^{2n}}\,dx
\end{eqnarray*}

Observe that for $ x \in \frac98Q$ and $y\in P\cap X_2(x)$,  by $d(x,\Omega)\le 4\sqrt n l_Q$ we have
$$|x-y|\le \frac1{132  n}\frac{132  n+1}{132  n-1}d(x,\Omega)\le l_Q .$$
By the assumption  \eqref{delta0} for  $\phi$, we have
 \begin{eqnarray*}
\int_{P\cap X_2(x)}\phi\left(\frac{|x-y|}{l_Q}\frac{|u_{Q^\ast}-u_{P^\ast }|}{\alpha/ 2L \gamma_0}\right)\, \frac{dy}{|x-y|^{2n}}
&&\le n\omega_n\int_0^{  l_Q} \phi\left(\frac{t}{l_Q}\frac{|u_{Q^\ast}-u_{P^\ast }|}{\alpha/ 2L \gamma_0}\right) \frac{dt}{t^{n+1}}\\
&&\le n\omega_n(  l_Q)^{-n } \left(\frac{|u_{Q^\ast}-u_{P^\ast }|}{ \alpha/ 2L \gamma_0}\right)^n\int_0^{ \frac{|u_{Q^\ast}-u_{P^\ast }|}{\alpha/ 2L \gamma_0}} \phi\left(s\right) \frac{ds}{s^{n+1}}\\
&&\le nC_\phi\omega_n(  l_Q)^{-n }  \phi\left( \frac{|u_{Q^\ast}-u_{P^\ast }|}{ \alpha/ 2 L  \gamma_0}\right).
\end{eqnarray*}
Using  the above inequality and \eqref{e2.w2}, one has
  \begin{eqnarray*}
H_{32}(\alpha)
 &&\le   nC_\phi  \frac1{  \gamma_0}  \omega_n\int_{U} \sum_{P\in \mathscr W}\sum \limits _{Q  \in N(P)}( l_Q)^{-n }  \chi_{\frac{9}{8}Q}(x)  \phi\left( \frac{|u_{Q^\ast}-u_{P^\ast }|}{4\alpha/ 2L\gamma_0}\right)
 \,dx\\
 &&\le  n C_\phi  \frac1{  \gamma_0} \omega_n  \sum_{P\in \mathscr W} \sum \limits _{Q  \in N(P) }\left(\frac1{|Q|}\int_U \chi_{\frac{9}{8}Q}(x) \,dx\right) \phi\left( \frac{|u_{Q^\ast}-u_{P^\ast }|}{ \alpha/ 2L  \gamma_0}\right)\\
 &&\le   C_\phi n \omega_n 4^{n}  \sum_{P\in \mathscr W}\sum \limits _{Q  \in N(P) }    \phi\left( \frac{|u_{Q^\ast}-u_{P^\ast }|}{ \alpha/ 2L \gamma_0}\right).
\end{eqnarray*}

For each  $P\in \mathscr W$ and $Q  \in N(P)$, by Jessen's inequality
$$
\phi\left( \frac{|u_{Q^\ast}-u_{P^\ast }|}{ \alpha/ 2 L  \gamma_0}\right) \le \fint_{Q^\ast}\fint_{P^\ast}\phi\left (\frac{|u(z)-u(w)|}{ \alpha/2 L \gamma_0 }\right)\,dz\,dw $$
Note that by Lemma 3.2 (i),    $ P^\ast\subset 10\sqrt n P$ and $ Q^\ast\subset 10\sqrt n Q$.
Thus for any $z\in P^\ast $ and $w\in Q^\ast $, by   $Q  \in N(P)$, we have
$$|z-w|\le 10\sqrt n(l_Q+l_P)\le 50  n \min\{l_Q, l_P\}.$$
Since $|Q|\le \gamma_1|Q^\ast|$ and   $|P|\le \gamma_1|P^\ast|$ as given in Lemma 3.2 (ii),  one gets
$$|z-w|^{2n}\le   (50 n)^{2n} (\gamma_1)^2 |Q^\ast|  |P^\ast| .$$
Therefore,
$$
\phi\left( \frac{|u_{Q^\ast}-u_{P^\ast }|}{ \alpha/ 2 L  \gamma_0}\right) \le    (50  n)^{2n} (\gamma_1)^2 \int_{Q^\ast}\int_{P^\ast}
 \phi\left( \frac{|u(z)-u(w)|}{ \alpha/2  L \gamma_0}\right)\,\frac{dz\,dw}{|z-w|^{2n}}$$
and hence
  \begin{eqnarray*}
H_{32}(\alpha)
  &&\le   C_\phi n \omega_n 4^{2n}  (50  n)^{2n} (\gamma_{1})^{2} \sum_{P\in \mathscr W}\sum \limits _{Q  \in N(P) }    \int_{Q^\ast}\int_{P^\ast}
 \phi\left( \frac{|u(z)-u(w)|}{ \alpha/2  L \gamma_0}\right)\,\frac{dz\,dw}{|z-w|^{2n}}
\end{eqnarray*}
With $\sum_{Q\in \mathscr W}\chi_{Q^\ast}\le \gamma_2$ as given in Lemma 2.2 (iii), we  obtain
   \begin{eqnarray*}
H_{32}(\alpha)
 &&\le  C_\phi n \omega_n 4^{2n} (50  n)^{2n} (\gamma_1)^2(\gamma_{2})^2  \int_{\Omega}\int_{\Omega}
 \phi\left( \frac{|u(z)-u(w)|}{ \alpha/ 2L \gamma_0}\right)\,\frac{dz\,dw}{|z-w|^{2n}}.
\end{eqnarray*}
Letting   $ M_ {32}=8L \gamma_0C_\phi n \omega_n 4^{2n} (50  n)^{2n} (\gamma_1)^2(\gamma_{2})^2  $.
If $\alpha>   M_ {32} $, we have $H_{32}(\alpha)\le 1/8$ as desired.

\medskip
\noindent
{\it Case $\diam\Omega<\fz$}.

To find $ M_ 2$,   write
 \begin{eqnarray*}
H_2(\alpha)
&&= \int_{V^{(2)}} \int_{\Omega} \phi \left(\frac{|Eu(x)-u(y)|}{\alpha}\right) \, \frac{dydx}{|x-y|^{2n}}+ \int_{U \backslash V^{(2)}} \int_{\Omega} \phi \left(\frac{| u(x)-u(y)|}{\alpha}\right) \, \frac{dydx}{|x-y|^{2n}} =H_{21}(\alpha) + H_{22}(\alpha).
\end{eqnarray*}
Recall that $V^{(2)}$ is defined by \eqref{vk.1}
 in Section 3.    It suffices to find $ M_ {2i}$ such that $H_{2i}\le 1/8$ for $i=1,2$.

Regards of $H_{22}(\alpha)$, observe that for any  $Q\in \mathscr W\setminus \mathscr W_{\ez_0}^{(2)}$,
  we  have $N(Q)\cap \mathscr W_{\ez_0}=\emptyset$, and hence  $P^\ast=\Omega$ for all $P\in N(Q)$.
Thus, for any $x\in U\setminus V^{(2)}$,   by Lemma 3.3 and Lemma 3.1  we have
$$Eu(x)=\sum_{P\in \mathscr W}\vz_P(x)u_{P^\ast}=\sum_{P\in N(Q)}\vz_P(x)u_{P^\ast}=u_\Omega.$$
Thus
 $$H_{22}(\alpha) =\int_{U \backslash V^{(2)}} \int_{\Omega} \phi \left(\frac{| u_\Omega-u(y)|}{\alpha}\right) \, \frac{dydx}{|x-y|^{2n}}.$$
By Jensen's inequality, one gets
 \begin{eqnarray*}
 H_{22}(\alpha)
 &&\leq  \int_{U \backslash V^{(2)}} \int_{\Omega}   \fint _{\Omega} \phi\left(\frac{|u(z)-u(y)|}{\alpha}\right) \,dz \frac{dydx}{|x-y|^{2n}}\\
  &&=  \int_{\Omega} \int_{U \backslash V^{(2)}}\frac{dx}{|x-y|^{2n}}   \fint _{\Omega} \phi\left(\frac{|u(z)-u(y)|}{\alpha}\right) \,dz dy \\
    &&=  \int_{\Omega} \left[\frac{|\diam \Omega|^{2n}}{|\Omega|}\int_{U \backslash V^{(2)}}\frac{dx}{|x-y|^{2n}} \right]  \int _{\Omega} \phi\left(\frac{|u(z)-u(y)|}{\alpha}\right) \frac{dz}{|z-y|^{2n}}.
   \end{eqnarray*}
For any $x\in U\setminus V^{(2)} $ and $y \in \Omega$, since there exists $Q\in \mathscr W\setminus \mathscr W_{\ez_0} $ so that $x\in Q$,
 one always has $$|x-y|\ge d(x,\Omega) \ge l_Q \ge \frac1{\ez_0}\diam\Omega.$$
Moreover, by the Ahlfors $n$-regular assumption,  it holds that  $ |\Omega|\ge C_A(\Omega)|\diam\Omega|^2$.
Thus,
 \begin{eqnarray*}
\frac{|\diam \Omega|^{2n}}{|\Omega|}\int_{U \backslash V^{(2)}}\frac{dx}{|x-y|^{2n}}
 &&\leq\frac1{ C_A(\Omega) } |\diam \Omega|^{n} \int_{|x-y|> \frac{1}{\epsilon_0}  \diam \Omega}\frac{dx}{|x-y|^{2n}} \\
 &&\leq |\diam \Omega|^{n} n\omega_n  \int^{\infty}_{\frac{1}{\epsilon_0} \diam \Omega} \frac{1}{r^{n+1}} dr \\
 && \leq  \omega_n \frac1 {C_A(\Omega)} \epsilon_0^{n},
 \end{eqnarray*}
from which, we conclude that
\begin{eqnarray*}
H_{22}(\alpha)&& \leq \omega_n \frac1 {C_A(\Omega)}   \epsilon_{0}^{n}\int_{\Omega} \int_{\Omega}  \phi\left(\frac{|u(z)-u(y)|}{\alpha}\right) \frac{dzdy}{|y-z|^{2n}}.
\end{eqnarray*}
Letting  $ M_ {22}=8 \omega_n \frac1 C_A(\Omega) \epsilon_{0}^{n}$,
 by the convexity of $\phi$ again,  for $\alpha >   M_ {22} $ we have $H_{22}(\alpha)\le 1/8$.

Regards of $H_{21}(\alpha)$,
observe that  $$\sum \limits _{Q \in \mathscr W } \varphi_Q (x)= \sum \limits _{Q \in \mathscr W_{\ez_0}^{(3)} } \varphi_Q (x) =1 $$
whenever  $x\in V^{(2)}$.
With aid of this and following, line by line, the
  argument to get \eqref{eq4.w2} for $H_2(\alpha)$  in the case $\diam\Omega=\fz$, one has
 \begin{eqnarray*}
H_{21}(\alpha)
&& \leq 2\gamma_1 4^n\gamma_0(21 n)^{2n}  \sum\limits_{Q \in \mathscr  W^{(2)}_{\ez_0}} \int _{Q^\ast}  \int_{\Omega}\phi\left(\frac{|u(z)-u(y)|}{\alpha}\right)   \frac{dzdy}{|y-z|^{2n}} .
\end{eqnarray*}
Here we omit the details. Since
 $$ \sum\limits_{Q \in \mathscr W^{(2)}_{\ez_0}}\chi_{Q^\ast}   \le \gamma_2+(\ez_0+128\sqrt{n})^n,$$
we have
 \begin{eqnarray*}
H_{21}(\alpha)
&& \leq 2\gamma_1 4^n\gamma_0  [\gamma_2+(\ez_0+ 64\sqrt{n})^n] (21 n)^{2n}  \int _\Omega  \int_{\Omega}\phi\left(\frac{|u(z)-u(y)|}{\alpha}\right)  \frac{dzdy}{|y-z|^{2n}}.
\end{eqnarray*}
Set $  M_ {21} =16\gamma_1 4^n\gamma_0  [\gamma_2+(\ez_0+ 64\sqrt{n})^n] (21 n)^{2n}$.
By the convexity of $\phi$ again, if $\alpha >   M_ {21} $, we have $H_{21}(\alpha)\le 1/8$ as dsired.

 To find $ M_ 3$, notice that
\begin{eqnarray*}U\times U
&\subset [V^{(3)}\times V^{(3)}]\cup  [V^{(2)}\times ( U \backslash  V^{(3)})] \cup  [ ( U \backslash  V^{(3)})\times V^{(2)} ] \cup [(U \backslash  V^{(2)} )\times (U \backslash  V^{(2)})] .
\end{eqnarray*}
We write
\begin{eqnarray*}
H_{3}&&=  \int_{V^{(3)}} \int_{ { V ^{(3)}}}   \phi \left(\frac{|Eu(x)-Eu(y)|}{\alpha}\right) \, \frac{dydx}{|x-y|^{2n}} +2  \int_{V^{(2)} } \int_{ U \backslash  V^{(3)}  }   \phi \left(\frac{|Eu(x)-Eu(y)|}{\alpha}\right) \, \frac{dydx}{|x-y|^{2n}}\\
&&\quad
+  \int_{U\setminus V^{(2)}} \int_{U\setminus V^{(2)}}   \phi \left(\frac{|Eu(x)-Eu(y)|}{\alpha}\right) \, \frac{dydx}{|x-y|^{2n}}\\
&&=: H_{31}(\alpha) +2H_{32}(\alpha)+H_{33}(\alpha).
\end{eqnarray*}

Since $Eu(x)=Eu(y)=u_\Omega$ for $x,y\in U\setminus V^{(2)}$, we have $H_{33}(\alpha)=0$.
It suffices to find $M_{3i}$   such that   $H_{3i}(\alpha)$ for all $\alpha>M_{3i}$ and $i=1,2$.

Regard of $H_{31}(\alpha)$, similarly to $H_{3}(\alpha)$ in the case $\diam\Omega=\fz$
 and taking $  M_ {31}$ as  $ M_ 3$ there with $\gamma_2$ replaced by $\gamma_2+(\ez_0+4^{5}\sqrt{n})^n$,
we can show that  if $\alpha\ge  M_ {31}$, then $H_{31}(\alpha)\le 1/8$. Here we omit the details.

For $H_{32}(\alpha)$,  note that for $y\in U \backslash   V^{(2)}$, we have $Eu(y)=u_\Omega$. Thus
 $$H_{32}(\alpha)= 2 \int_{V^{(2)}} \int_{ U \backslash   V^{(3)}  }   \phi \left(\frac{|Eu(x)-u_\Omega|}{\alpha}\right) \, \frac{dydx}{|x-y|^{2n}}.$$
 By Jessen's inequality, one has
 $$H_{32}(\alpha)\le   \int_{V^{(2)}} \int_{ U \backslash   V^{(3)}  } \, \frac{dy }{|x-y|^{2n}}\fint_\Omega \phi \left(\frac{|Eu(x)-u(z)}{\alpha}\right) \,   dx\,dz  .$$
For any $x\in V^{(2)}$ and $y\in U \backslash   V^{(3)}$  note that
 $|x-y|\ge l(Q)\ge \frac1 {\ez_0}\diam \Omega$,
 where $Q\in \mathscr W_{\ez_0}^{(3)}\setminus \mathscr W_{\ez_0}^{(2)}$ and $y\in Q$.
 Thus
 $$\int_{ U \backslash   V^{(3)}  } \, \frac{dy }{|x-y|^{2n}}\le \ez_0^n (\diam \Omega)^{-n}.$$
Since $|\Omega|\ge C_A(\Omega)\diam\Omega$,  one   has
  $$H_{32}(\alpha)\le  \frac1{C_A(\Omega)}\ez_0^{ n} (\diam \Omega)^{-2n}    \int_{V^{(2)}}  \int_\Omega \phi \left(\frac{|Eu(x)-u(z)|}{\alpha}\right) \,   dx\,dz .$$
  Note that
  for any $x\in V^{(2)}$  there exists a $P_i\in \mathscr W^{(i)}_{\ez_0}$ such that
  $x\in P_2$ and $P_i\in N(P_{i-1})$ for $i=1,2$.
  Since $l(P_0)\le \frac1{\ez_0}\diam\Omega$, by Lemma 3.1 we know that
  $l(P_2)\le 4^2 \frac1{\ez_0}\diam\Omega$.
 Thus for $y\in \Omega$, one has
  $$|x-y|\le \dist(x,\Omega)+\diam\Omega\le  \diam P_2+\dist(P_2,\Omega)+ \diam\Omega
  \le  4^4 \frac1{\ez_0}\sqrt n \diam\Omega.$$
  Therefore,
   $$H_{32}(\alpha)\le  \frac1{C_A(\Omega)}\ez_0^{ n}  (4^4 \frac1{\ez_0}\sqrt n)^{2n}    \int_{V^{(2)}}  \int_\Omega \phi \left(\frac{|Eu(x)-u(z)|}{\alpha}\right) \,   \frac{dx\,dz}{|x-z|^{2n}}\le
     \frac1{C_A(\Omega)}    4^ {8n}  \ez_0^{-n}  n ^{ n} H_{21}(\alpha).$$
If $\alpha>M_{32}=  \frac8{C_A(\Omega)} 4^ {8n}  \ez_0^{-n}  n ^{ n}  M_ {21}$, we have $H_{32}(\alpha)\le 1/8$.
This completes the proof of Theorem 1.1 (i).

\begin{rem}\rm
We emphasis that the bounded overlaps of reflecting cubes $Q^\ast$  in Lemma 3.2 (iii) play central roles  in the proof of the boundedness of extension operator $E: \dot{\bf B}^{\phi} (\Omega)$ to $\dot {\bf B}^\phi(\rn)$.

Similarly to \cite{z14,hkt08,k98} and the reference therein,
one may define the extension operator $\wz E u$ similarly to  $Eu$ but replacing $Q^\ast$ in $Eu$ with
$Q(x_Q^\ast,l_Q)\cap\Omega$, where  $x^\ast_Q$ is the nearest point in $\Omega$ of $Q   \in\mathscr W $.
Note that $  \{Q(x_Q^\ast,l_Q)\cap\Omega, Q   \in\mathscr W \}$ does not have bounded overlap property as in Lemma 3.2(iii) in general.

 In the case $\phi(t)=t^p$ with $p>n$, similarly to \cite{z14}, one may prove that $\wz E  $ is bounded from $\dot {\bf B}^{n/p}_{pp}(\Omega)$ to $\dot {\bf B}^{n/p}_{pp}(\rn)$.  The point is prove that
$$\frac{|\wz Eu(x)-\wz Eu(y)|^p}{|x-y|^{2n}} \le \mathcal M\left(\frac{|u(z)-u(w)|^p}{|z-w|^{2n}} \chi_{\Omega\times\Omega}\right)(x,y)$$  where $\mathcal M$ is certain Hardy-Littlewood maixmal operator.     See page 968 in the proof of \cite[Theorem 1.1]{z14}.

For general $\phi$ in Theorem 1.1, some appropriate  estimates of
of $\phi(\frac{\wz Eu(x)-\wz Eu(y)}\alpha) \frac1{|x-y|^{2n}} $ via certain maximal functions are not available for us. We do not know if it is possible to obtain the boundedness of $\wz E$ from $ {\bf B}^{\phi} (\Omega)$ to $\dot {\bf B}^\phi(\rn)$. Note the our proof of the boundedness of $E$ does not work for $\wz E$ since $  \{Q(x_Q^\ast,l_Q)\cap\Omega, Q   \in\mathscr W \}$  does not have the bounded overlap property.

\end{rem}

\section{Proof of Theorem \ref{t1.1} (ii)}

We divide the proof into 3 steps.

\noindent {\it  Step 1.}
Since $\Omega$ is a ${\dot {\bf B}}^\phi$-extension domain,
there exists a bounded linear extension operator $E: {\dot {\bf B}}^\phi(\Omega)\to {\dot {\bf B}}^\phi(\rn)$.
For any $u\in {\dot {\bf B}}^\phi(\Omega)$,   we  have  $Eu  \in {\dot {\bf B}}^\phi(\rn)$ with
 $E u=u$ in $\Omega$, $\|E u \|_{{\dot {\bf B}}^\phi(\rn)  }\le   \|E\|\|u\|_{{\dot {\bf B}}^\phi(\Omega)}$.
 By Lemma 2.3, we have $\wz u\in BMO(\rn)$ and
 $\|E u \|_{BMO(\rn) }\le \phi^{-1}({|S^1|^2})\|E u\|_{{\dot {\bf B}}^\phi(\rn)}$.
From the John-Nirenberg inequality, it follows that
 $$
 \int_{B} \exp \left(\frac{|E u-(E u)_{B}|}{C_{JN}(n)\|E u\|_{BMO(\rn)}}\right)\,dx \le   C(n)|B|\quad \mbox{for all balls  $B \subset \rn$ }
 $$
Thus,
\begin{equation}\label{IM}
 \inf_{c\in\rr}\int_{B\cap \Omega} \exp \left(\frac{| u-c|}{C(\phi,n,\Omega)\|u\|_{{\dot {\bf B}}^\phi(\Omega)}}\right)\,dx \le   C(n)|B|\quad \mbox{for all balls  $B \subset \rn$.}
\end{equation}

\noindent {\it Step 2.}
For $x\in \Omega$ and $0<r<t<\diam\Omega$, set  the function
$$u_{x,r,t}(z)=\left\{\begin{array}{ll}1& z\in B(x,r) \cap \Omega\\
\frac{t-|x-z|}{t-r}\quad & z\in \left(B(x,t)\setminus B(x,r)\cap\right) \Omega\\
0& z\in\Omega\setminus B(x,t)
\end{array}\r.$$
We have the following.
\begin{lem}\label{l2.3}
Suppose that $\phi$ is a Young  function satisfying \eqref{delta0}.
For $x\in \Omega$ and $0<r<t<\diam\Omega$,  we have $u_{x,r,t}\in {\dot {\bf B}}^\phi(\Omega)$ with
$$\|u_{x,r,t}\|_{{\dot {\bf B}}^\phi(\Omega)}\le 8\omega_n[C_\phi 4^n+1] \left [\phi^{-1}\left(\frac{(t-r)^n}{| B(x,t) \cap \Omega| }\right)\right]^{-1}.$$
\end{lem}

\begin{proof}[Proof of Lemma \ref{l2.3}]
Write
\begin{align*}
 & \int_\Omega\int_\Omega\phi\left(\frac{|u_{x,r,t}(z)-u_{x,r,t}(w)|}\alpha\right)\frac{dzdw}{|z-w|^{2n}}\\
 &\quad=\int_{B(x,t)\cap \Omega}\int_{B(x,t)\cap \Omega}\phi\left(\frac{|u_{x,r,t}(z)-u_{x,r,t}(w)|}\alpha\right)\frac{dzdw}{|z-w|^{2n}}+\int_{\Omega\setminus B(x,t)}\int_{B(x,t)\cap \Omega}\phi\left(\frac{|  u_{x,r,t}(z)|}\alpha\right)\frac{dzdw}{|z-w|^{2n}}\\
&\quad=:H_1(\alpha)+H_2(\alpha).
\end{align*}
If suffices to find  a constant $M$ depending only on $n$ such that for
$\alpha= M[\phi^{-1}\left(\frac{(t-r)^n}{|B(x,t)\cap \Omega| }\right)]^{-1}$, we have $H_1\le \frac12$ and
$H_2(\alpha)\le \frac12 $.

Write
\begin{align*}H_1(\alpha)&\le \int_{B(x,t)\cap \Omega}\int_{B(w,t-r)\cap \Omega}\phi\left(\frac{|z-w|}{\alpha(t-r)}\right)\frac{dz}{|z-w|^{2n}}dw+\int_{B(x,t)\cap \Omega}\int_{\left(B(z,2t)\setminus B(w,t-r)\right)\cap \Omega}\phi\left(\frac{ 1}{\alpha }\right)\frac{dz}{|z-w|^{2n}}dw.
\end{align*}
Observe that
\begin{align*} \int_{B(w,t-r)\cap \Omega}\phi\left(\frac{|z-w|}{\alpha(t-r)}\right)\frac{dz}{|z-w|^{2n}}
&\le n\omega_n \int_0^{t-r}  \phi\left(\frac{s}{\alpha(t-r)}\right)\frac{ds}{s^{n+1}}  \le n\omega_n (t-r)^{-n}\alpha^{-n}\int_0^{1/\alpha}  \phi\left(s\right)\frac{ds}{s^{n+1}}.
\end{align*}
Applying \eqref{delta0}, we have
\begin{align*} \int_{B(z,t-r)\cap \Omega}\phi\left(\frac{|z-w|}{\alpha(t-r)}\right)\frac{dz}{|z-w|^{2n}}
&\le n C_\phi\omega_n ( t-r)^{-n}  \phi\left(\frac{ 1  }{\alpha }\right) .
\end{align*}
On the other hand,
\begin{align*} \int_{\left(B(w,2t)\setminus B(w,t-r)\right)\cap \Omega}  \frac{dz}{|z-w|^{2n}}  &\le \int_{\rn \setminus B_ (w,t-r)}  \frac{dz}{|z-w|^{2n}} =  \omega_n (t-r)^{-n}
\end{align*}
Thus
$$H_1(\alpha)\le \omega_n(n C_\phi +1)\frac{|B(x,t)\cap \Omega|}{(t-r)^n}\phi\left(\frac{1 }{\alpha }\right).$$
If $\alpha=M[\phi^{-1}\left(\frac{(t-r)^n}{|B(x,t)\cap \Omega| }\right)]^{-1}$ and $M\ge 2(nC_\phi+1)\omega_n$, we have
$$H_1(\alpha)\le \frac{(n C_\phi+1) \omega_n}{M}\le 1/2 .$$

Write
\begin{align*}H_2(\alpha)&\le \int_{\left(B(x,t) \setminus B(x,r)\right)\cap \Omega}\phi\left(\frac{t-|z-x|}{\alpha(t-r)}\right)\int_{\Omega\setminus B(x,t)}\frac{dw}{|z-w|^{2n}}dz + \int_{B(x,r) \cap \Omega}
\int_{\Omega\setminus B(x,t)}\phi\left(\frac{1}{\alpha }\right)\frac{dw}{|z-w|^{2n}}dz .
\end{align*}
Note that $ \Omega\setminus B(x,t)\subset \Omega\setminus B(z, t-|z-x|)$,
we have
$$\int_{\Omega\setminus B(x,t)}\frac{dw}{|z-w|^{2n}}\le \int_{\rn\setminus B(z, t-|z-x|)}\frac{dw}{|z-w|^{2n}}\le \omega_n  (t-|z-x|)^{-n}$$
Hence,
\begin{align*}H_2(\alpha)&\le \int_{\left(B(x,t) \setminus B(x,r)\right)\cap \Omega}\phi\left(\frac{t-|z-x|}{\alpha(t-r)}\right) \omega_n  (t-|z-x|)^{-n}dz + \int_{B(x,r) \cap \Omega}
 \phi\left(\frac{1}{\alpha }\right) \omega_n  (t-|z-x|)^{-n}dz \\
&\le  2 \omega_n\frac{|B(x,t)\cap \Omega|}{(t-r)^n}\left[
\sup_{s\in(0,1]}\phi\left(\frac{s}{\alpha }\right)\frac1{s^{ n}}  +\phi\left(\frac{1 }{\alpha }\right)\right]\\
\end{align*}
Notice that
$$
\sup_{s\in(2^{-j-1},2^{-j}]}\phi\left(\frac{s}{\alpha }\right)\frac1{s^{ n}}\le  2^n \int_{2^{-j}}^{2^{-j+1}} \phi\left(\frac{s}{\alpha }\right)\frac{ds}{s^{ n+1}}  $$
an hence
$$\sup_{s\in(0,1]}\phi\left(\frac{s}{\alpha }\right)\frac1{s^{ n}}\le 2^n\int_0^2  \phi\left(\frac{s}{\alpha }\right)\frac{ds}{s^{ n+1}}\le  2^n\alpha^{-n}\int_0^{2/\alpha}  \phi\left(s \right)\frac{ds}{s^{ n+1}}\le C_\phi 4^n\phi\left(\frac{1}{\alpha/2 }\right)$$
Therefore,
\begin{align*}H_2(\alpha)
&\le  2(C_\phi4^n+1) \omega_n\frac{|B(x,t)\cap \Omega|}{(t-r)^n} \phi\left(\frac{2 }{\alpha } \right).
\end{align*}
If $\alpha=M[\phi^{-1}\left(\frac{(t-r)^n}{|B(x,t)\cap \Omega| }\right)]^{-1}$ and $M\ge 8  (C_\phi  4^n +1) \omega_n$, we have
$$H_2(\alpha)\le \frac{2 (C_\phi4^n+1)\omega_n}{M/2}\le \frac12. $$
as desired.
\end{proof}

 {\it Step 3.}
Let $x\in\Omega$ and $ 0<r<2\diam\Omega$.
Let $b_0=1$ and $b_j \in (0,1)$ for $j \in N$ such that
\begin{equation}\label{e5.w1}
|B(x, b_jr) \cap \Omega|= 2^{-1}|B(x, b_{j-1}r) \cap \Omega|= 2^{-j}|B(x, r) \cap \Omega|.
\end{equation}
Let  $u_j=u_{x,b_{j+1 }r, b_{j }r} $ for $j\ge 1$ be as in Lemma \ref{l2.3}.
 By \eqref{IM}, we have
$$
\inf_{c\in R}  \int_{B(x, b_{j-1} r) \cap \Omega} \exp\left(  \frac{|u_j-c| }{C (\phi,n,\Omega)\|u\|_{{\dot {\bf B}}^\phi(\Omega)}}\right)\, dy \,\le C(n)  r^{n};
$$
For any       $c \in R $,  we know that $|u_j-c| \geq 1/2$ either on $B(x,\,b_{j+1} r\,) \cap \boz$ or on $[ B(x,b_{j-1} r)\diagdown B(x,\, b_{j }r\,) ]\cap \Omega$, and note that, by \eqref{e5.w1},
$$|B(x,\,b_{j+1} r\,) \cap \boz|= |[ B(x,b_{j-1} r)\diagdown B(x,\, b_{j }r\,) ] \cap \Omega|=2^{-j-1}|B(x,\,  r\,) \cap \boz| .$$
Thus, for any  $j\ge1$,
 we have
$$
 2^{-j-1}|B(x,\,  r\,) \cap \boz| \exp\left(  \frac{|u_j-c| }{C (\phi,n,\Omega)\|u_j\|_{{\dot {\bf B}}^\phi(\Omega)}}\right)
 \le C(n) r^{n}
$$
that is,
$$\frac{|u_j-c| }{C (\phi,n,\Omega)\|u_j\|_{{\dot {\bf B}}^\phi(\Omega)}}\le  \ln \left( 2^{j} \frac{C(n)r^n}{ \left| B(x,\,r\,) \cap \boz\right|} \right)  .$$
Since
\begin{eqnarray*}
\|u_j\|_{{\dot {\bf B}}^\phi(\Omega)} \leq C(\phi,n)\left [\phi^{-1}\left(\frac{(b_{j} r - b_{j+1} r)^n}{|B(x , b_jr) \cap \Omega|}\right)\right]^{-1} = C(\phi,n)\left [\phi^{-1}\left(2^{j}\frac{(b_{j} r - b_{j+1} r)^n}{|B(x ,r) \cap \Omega|}\right)\right]^{-1},
\end{eqnarray*}
we have
$$
\frac1{  C(\phi,n,\Omega)}\phi^{-1}\left(2^{j}\frac{(b_j r - b_{j+1} r)^n}{|B(x , r) \cap \Omega|}\right)\le  \ln \left( 2^{j} \frac{C(n)r^n}{ \left| B(x,\,r\,) \cap \boz\right|} \right),$$
and hence
\begin{align*}
(b_j  - b_{j+1})^n  & \leq 2^{-j }\frac{|B(x, r) \cap  \Omega|} {r^n}
\phi \left[ { C(\phi,n,\Omega)}\ln \left( 2^{j}\frac{2C(n)r^n}{\left| B(x,\, r\,) \cap \boz\right|} \right)\right]
\end{align*}

By \eqref{subexp}, for any $\dz>0$, we have  $\phi(t)\le C(\dz) e^{\dz  t}$ for all $t\ge 0$.
Taking   $\dz_0= 1/2 C(\phi,n,\Omega)$, that is, $C(\phi,n,\Omega)\dz_0=1/2$,  we obtain
\begin{align*}
(b_j  - b_{j+1})^n  & \leq  C(\dz_0)[2C(n)]^{\dz_0 C(\phi,n,\Omega)}\left ( 2^{-j}\frac{|B(x, r) \cap  \Omega| }{r^n}\right)  ^{1-  { \dz_0 C(\phi,n,\Omega) } }\\
&\le   C(\phi,n,\Omega)\left (  \frac{|B(x,  r) \cap  \Omega| }{r^n}\right)  ^{1/2 } 2^{-j/2 }.
\end{align*}
 Thus
   \begin{align*}
b_1  \, = \sum_{j=1}^{\infty} \left(b_{j }  - b_{j+1} \right)
\le    C(\phi,n,\Omega)\left (  \frac{|B(x,  r) \cap  \Omega| }{r^n}\right)  ^{1/2n } \end{align*}

If $b_1 \geq 1/10 $,    we get
\begin{eqnarray*}
\left|B(x,r) \cap \Omega\right| \ge C(\phi,n,\Omega) r^n
\end{eqnarray*}
as desired.

If $b_1 < 1/10 $, we can know that exists a point $x' \in B(x,r) \cap \Omega$ satisfying $|x-x'|= b_1 r+ r/5$. Let $R=2r/5$, then $B(x, b_1 r) \subset B(x',R) \subset B(x,r)$ and $B(x, b_1 r) \cap B(x',R/2) =\emptyset $.
Thus
\begin{eqnarray*}
\left|B(x',R/2) \cap \Omega \right|
&& \leq  \frac{1}{2}\left( \left|\left(B(x,r) \backslash B(x, b_1r)\right) \cap \Omega\right|+\left|B(x',R/2) \cap \Omega \right| \right)\\
&& =\frac{1}{2}\left( \left|B(x,b_1 r) \cap \Omega\right|+\left|B(x',R/2) \cap \Omega \right| \right)\\
&& \leq \frac{1}{2}\left|B(x',R ) \cap \Omega \right|,
\end{eqnarray*}
By this, if
$\left|B(x',  b'_1 R) \cap \Omega \right|=\frac{1}{2}\left|B(x',R ) \cap \Omega \right|$,
then $b_1'\ge 1/2$.
Applying the result when $b_1 \geq 1/10 $ to the $B(x',R/2)$ and $b_1'\ge 1/2$, we get
$$
\left|B(x,r) \cap \Omega\right| \geq \left|B(x',R)\cap \Omega\right| \geq C(\phi,n,\Omega)  R^{n}\ge C(\phi,n,\Omega) r^n ,
$$  as desired. This completes the proof of Theorem 1.1 (ii).

\medskip



\end{document}